\documentclass[a4paper,12pt]{article}

\usepackage{latexsym}
\usepackage{mathrsfs}
\usepackage{amsfonts,amsmath,amssymb}
\usepackage{indentfirst}
\usepackage{subeqnarray}
\usepackage{verbatim}
\usepackage{epstopdf}
\usepackage{algorithm}
\usepackage{algorithmic}
\usepackage{graphicx}
\usepackage{subfigure}
\usepackage{color}
\usepackage[pagewise]{lineno}

\newtheorem{theorem}{Theorem}[section]
\newtheorem{lemma}[theorem]{Lemma}
\newtheorem{assumption}[theorem]{Assumption}

\newtheorem{definition}[theorem]{Definition}

\newtheorem{proposition}[theorem]{Proposition}

\newtheorem{condition}[theorem]{Condition}
\numberwithin{equation}{section}
\newenvironment{proof}[1][Proof]{\textbf{#1.} }
{\ \rule{0.75em}{0.75em}\smallskip}

\begin{document}

\begin{center}
\Large\bf On existence of a variational regularization parameter under Morozov's discrepancy principle
\end{center}

\begin{center}
Liang Ding\footnote{Department of Mathematics, Northeast Forestry University, Harbin 150040, China;
e-mail: {\tt dl@nefu.edu.cn}. The work of this author was supported by the Fundamental Research Funds
for the Central Universities (no.\ 2572021DJ03).}\quad 
Long Li\footnote{Department of Mathematics, Northeast Forestry University, Harbin 150040, China;
e-mail: {\tt 15146259835@nefu.edu.cn}.}\quad
Weimin Han\footnote{Department of Mathematics, University of Iowa, Iowa City, IA 52242, USA;
e-mail: {\tt weimin-han@uiowa.edu}.}\quad
Wei Wang\footnote{Department of Mathematics, Jiaxing University, Jiaxing 314001, China;
e-mail: {\tt weiwang@zjxu.edu.cn}.}
\end{center}

\smallskip
\begin{quote}
{\bf Abstract.}
{Morozov's discrepancy principle is commonly adopted in Tikhonov regularization for choosing the regularization parameter. Nevertheless, for a general non-linear inverse problem, the discrepancy $\|F(x_{\alpha}^{\delta})-y^{\delta}\|_Y$ does not depend continuously on $\alpha$ and it is questionable whether there exists a regularization parameter $\alpha$ such that $\tau_1\delta\leq \|F(x_{\alpha}^{\delta})-y^{\delta}\|_Y\leq \tau_2 \delta$ $(1\le \tau_1<\tau_2)$. In this paper, we prove the existence of $\alpha$ under Morozov's discrepancy principle if $\tau_2\ge (3+2\gamma)\tau_1$, where $\gamma>0$ is a parameter in a tangential cone condition for the nonlinear operator $F$. Furthermore, we present results on the convergence of the regularized solutions under Morozov's discrepancy principle. Numerical results are reported on the efficiency of the proposed approach.}

\end{quote}


\section{Introduction}\label{sec1}

\par In this paper, we consider to solve an operator equation of the form
\begin{equation}\label{equ1.1}
    F(x)=y,
\end{equation}
where $F: X\rightarrow Y$ is a nonlinear operator between two reflexive Banach spaces $X$ and $Y$. We assume $y^{\delta}\in Y$ is known with $\|y^{\delta}-y\|_{Y}\leq \delta$ for a given $\delta\geq 0$. The commonly adopted technique to solve problem \eqref{equ1.1} is the Tikhonov regularization of the form
\begin{equation}\label{equ1.2}
    \min\limits_{x\in X}\bigg\{\mathcal{J}_{\alpha}^{\delta}(x)= \frac{1}{2}\|F(x)-y^{\delta}\|_Y^{2}+\alpha\mathcal{R}(x)\bigg\},
\end{equation}
where $\alpha>0$ and ${\cal R}$ is a convex penalty term, cf.\ the monographs \cite{F2010,SGGHL2009} and the special issues \cite{BB18,DDD16,JM12,JMS17} for many developments on regularization properties and minimization schemes.

\par In practice, it is crucial to choose an appropriate regularization parameter $\alpha$ for problem \eqref{equ1.2}. Morozov's discrepancy principle (MDP) is the most commonly adopted technique to determine the regularization parameter $\alpha=\alpha(\delta, y^{\delta})>0$ such that
\begin{equation}\label{equ1.3}
    \tau_1\delta\leq \|F(x_{\alpha}^{\delta})-y^{\delta}\|_Y\leq \tau_2 \delta,
\end{equation}
where $1\le\tau_1<\tau_2$ and $x_{\alpha}^{\delta}$ is a minimizer of $\mathcal{J}_{\alpha}^{\delta}(x)$ in \eqref{equ1.2}. If the minimizer of $\mathcal{J}_{\alpha}^{\delta}(x)$ is unique, there exists a regularization parameter $\alpha$ such that \eqref{equ1.3} holds (\cite{AD19,AR10,TA1977}). However, due to the non-linearity of $F$, there may exist multiple minimizers of $\mathcal{J}_{\alpha}^{\delta}(x)$, the discrepancy $\|F(x_{\alpha}^{\delta})-y^{\delta}\|_Y$ does not exhibit continuous dependence on $\alpha$ and one can not ensure the existence of $\alpha$ such that \eqref{equ1.3} holds (\cite{R02, TA1977}). Indeed, a weakness of MDP lies in the fact that a regularization parameter $\alpha$ satisfying \eqref{equ1.3}
might not exist for a general nonlinear operator equation. In some applications of non-linear inverse problems, one has to assume the existence of $\alpha$ such that \eqref{equ1.3} holds. Few results are available on the existence of $\alpha$ under MDP.

\subsection{Related works}\label{sec1.1}

\par The first theoretical analysis on MDP for the Tikhonov regularization dates back to 1966 (\cite{M66}). Several numerical algorithms have been proposed to compute the regularization parameter $\alpha$ for the classical quadratic Tikhonov regularization (\cite{FM99,L1989,MPRS01}). Subsequently, MDP is extended to general convex regularizations. If $\mathcal{R}(x)$ is convex, then $\left\|F\left(x_{\alpha}^{\delta}\right)-y^{\delta}\right\|_{Y}$ depends continuously on $\alpha$, and the existence of $\alpha$ follows. In \cite{B09}, MDP is applied to the Tikhonov regularization with convex penalty terms  more general than the classical quadratic one. In \cite{JZZ12}, two iterative parameter choice methods by MDP are proposed for non-smooth Tikhonov regularization with general convex penalty terms.

\par We note that the results cited in the previous paragraph are limited to linear ill-posed problems. Due to the existence of possibly multiple minimizers, the techniques for linear ill-posed problems can not be extended to nonlinear ill-posed operator equations directly. Special regularization techniques are needed to explore the existence of the regularization parameter $\alpha$ which is determined by MDP. In \cite{R02}, for the classical quadratic Tikhonov regularization, it is shown that if $\tau_1=1$ and
    \[ \tau_2> 1+\frac{L+2\|F'({x^{\dag}})\|}{L}\max\left\{\frac{4}{Q^2}, 1\right\}, \]
the existence of $\alpha$ can be guaranteed, where $x^{\dag}$ is an $\mathcal R$-minimum solution of the equation \eqref{equ1.1}, and the parameters $L$ and $Q$ depend on the source condition and the lower bound of $x^{\dag}$. In \cite{AR10}, for a general convex penalty term, the existence of $\alpha$ is proved under the following condition:
there is no $\alpha>0$ with minimizers $x_{\alpha}^{\delta,1},x_{\alpha}^{\delta,2}$ such that
    \[\|F(x_{\alpha}^{\delta,1})-y^{\delta}\|_{Y}<\tau_1\delta <\tau_2\delta<\|F(x_{\alpha}^{\delta,2})-y^{\delta}\|_{Y}.\]
In \cite{AR11}, convergence rates are investigated via variational inequalities, where the regularization parameter is determined by MDP. A relation between uniqueness of minimizers in Tikhonov-type regularization and Morozov-like discrepancy principles is shown in \cite{AD19}. In \cite{DH23}, the convex function $\mathcal{R}(x)$ is replaced by a non-convex penalty term $\alpha\ell_{1}(x)-\beta\ell_{2}(x)$, and the existence of the regularization parameter $\alpha$ is demonstrated. It is proven that there exists at least one $\alpha$ such that
\begin{equation}\label{equ1.4}
    \delta\le\|F(x_{\alpha,\beta}^{\delta})-y^{\delta}\|_{Y}\le \left(\max\left\{ \tau^{2}\delta^2,(3+2\gamma)\delta^{2}+\delta(2+2\gamma) \|y^{\delta}-F(0)\|_Y\right\}\right)^{\frac{1}{2}}.
\end{equation}
However, \eqref{equ1.4} does not give a same order for the lower and the upper bound. The upper bounded is only $O(\sqrt{\delta})$.

\subsection{Contribution and organization}

\par In this paper, we assume $\mathcal{R}$ is convex and the operator $F$ satisfies a tangential cone condition: there exists a constant $\gamma>0$ such that
for arbitrary $x_1, x_2\in \mathcal{L}_{\alpha}^{\delta}$,
    \[ \|F(x_2)-F(x_1)-F'(x_1)(x_2-x_1)\|_Y\leq \gamma\|F(x_2)-F(x_1)\|_Y,\]
where $\mathcal{L}_{\alpha}^{\delta}$ represents the set of all minimizers of the functional $\mathcal{J}_{\alpha}^{\delta}\left(x\right)$ defined in \eqref{equ1.2}. We will prove that there exists at least one $\alpha$ such that
\begin{equation}\label{equ1.5}
    \tau_1\delta\le\|F(x_{\alpha}^{\delta})-y^{\delta}\|_{Y}\le c\delta,
\end{equation}
where $c:=\max\left\{ \tau_2,\tau_1(3+2\gamma)\right\}$. This implies that if $\tau_2\ge \tau_1(3+2\gamma)$, the existence of $\alpha$ can be guaranteed.
Based on \eqref{equ1.5}, we will investigate the well-posedness and convergence rate of the regularized solution of the problem \eqref{equ1.2}. We will show that with the parameter choice rule \eqref{equ1.5}, $\alpha\rightarrow 0$ as the noise level $\delta\rightarrow 0$.

\par For the traditional MDP for linear ill-posed problems, the existence of $\alpha$ is guaranteed. However, for nonlinear ill-posed problems, the discrepancy term $\|F(x_{\alpha}^{\delta})-y^{\delta}\|_{Y}$ is not continuous with respect to $\alpha$. So it is challenging to choose appropriate values of $\tau_1$ and $\tau_2$ to ensure the existence of $\alpha$. We will prove that the modified MDP \eqref{equ1.5} can ensure the existence of $\alpha$ and we can apply Algorithm \ref{alg1} to determine the value of $\alpha$.
We will present results from two numerical experiments to show that the upper bound $c\delta$ in \eqref{equ1.5} is actually needed, cf.\ Fig.\ \ref{fig6} and Fig.\ \ref{fig10}.

An outline of the rest of this paper is as follows. In the next section we introduce the notation and review some results on the Tihkonov regularization with a general convex penalty term. In Section \ref{sec3}, we investigate the existence of $\alpha$ determined by MDP \eqref{equ1.5}. In Section \ref{sec4}, we present results on the well-posedness and the convergence rate of the regularized solution under MDP \eqref{equ1.5}. In addition, we show that MDP \eqref{equ1.5} can guarantee $\alpha(\delta,y^{\delta})\rightarrow0$ as $\delta\rightarrow 0$. Finally, numerical results are reported in Section \ref{sec5} on a nonlinear compressive sensing problem and a nonlinear one-dimensional gravity inversion problem.

\section{Preliminaries}\label{sec2}

\par We first briefly introduce some notation for the Tikhonov regularization. Denote by
\begin{equation}\label{equ2.1}
    x^{\delta}_{\alpha}\in \mathop{\arg\min}_{x\in X}\left\{\frac{1}{2} \|F(x)-y^{\delta}\|_Y^{2}+\alpha\mathcal{R}(x)\right\}
\end{equation}
a minimizer of the regularization function $\mathcal{J}_{\alpha}^{\delta}(x)$ defined
in (\ref{equ1.2}). Let $\mathcal{L}^{\delta}_{\alpha}$ be the set of all such minimizers $x^{\delta}_{\alpha}$.
An element $x^{\dagger}\in X$ is called an $\mathcal R$-minimum solution of the equation \eqref{equ1.1} if
    \[ F(x^{\dagger})=y~~and~~ \mathcal R(x^{\dagger})
    =\min\limits_{x\in X}\{\mathcal R(x)\mid \, F(x)=y\}.\]

\par Throughout this paper, we will assume the operator $F$ and the data $y^\delta$ have the following properties.

\begin{condition}\label{con2.1}
{\rm(i)} $F: X \rightarrow Y$ is Fr\'{e}chet differentiable.

{\rm(ii)} $F: X \rightarrow Y$ is weakly sequentially closed, i.e.\ $x_n\rightharpoonup x$ in $X$ and $F(x_n)\rightharpoonup y$ in $Y$ imply that $F(x)=y$.

{\rm(iii)} There exists a constant $\gamma>0$ such that
\begin{equation}\label{equ2.2}
\|F(x_2)-F(x_1)-F'(x_1)(x_2-x_1)\|_Y\leq \gamma\,\|F(x_2)-F(x_1)\|_Y\quad\forall\, x_1, x_2\in \mathcal{L}_{\alpha}^{\delta}.
\end{equation}

{\rm(iv)} There exist $\delta>0$ and $\tau_2>\tau_1\ge 1$ such that
\begin{equation}\label{equ2.3}
\|y-y^{\delta}\|_Y\le \tau_1\delta<\tau_2\delta\le\|F(0)-y^{\delta}\|_Y.
\end{equation}
\end{condition}

The condition \eqref{equ2.2} was used in \cite[pp.\ 278--279]{EHN1996}, \cite[p.\ 6]{KNS2012}, \cite[pp.\ 69--70]{SKHK2012}. It follows from \eqref{equ2.2} that
    \[\|F'(x_1)(x_2-x_1)\|_Y\leq(1+\gamma)\|F(x_2)-F(x_1)\|_Y,\]
which has been adopted by several researchers. In \eqref{equ2.3}, we require $\|F(0)-y^{\delta}\|_Y\ge \tau\delta$, which is a reasonable assumption. Indeed, if $\|F(0)-y^{\delta}\|_Y< \tau\delta$, then $0$ can be viewed as a good approximation to the regularized solution $x_{\alpha}^{\delta}$. Moreover, practically, it is almost impossible to recover a solution from an observed data of a size in the same order as the noise.

\par Next, we assume the convex penalty term $\mathcal{R}$ has the following properties.

\begin{condition}\label{con2.2}

{\rm(i)}  $\mathcal{R}(x)$ is proper, and $\mathcal{R}(x)=0$ if and only if $x=0$.

{\rm(ii)  (Coercivity)} $\|x\|_{{X}}\rightarrow \infty$ implies $\mathcal{R}(x)\rightarrow \infty$.

{\rm(iii)  (Weak lower semi-continuity)} $x_n\rightharpoonup x$ in $X$ implies
$ \liminf_n\mathcal{R}(x_n)\geq\mathcal{R}(x)$.

{\rm (iv) (Radon-Riesz property)} $x_n\rightharpoonup x$ in $X$ and $\mathcal{R}(X_n)\rightarrow
\mathcal{R}(X)$ imply $x_n\rightarrow x$ in $X$.
\end{condition}

\begin{definition}\label{def2.3}
{\rm (Bregman distance)} A Bregman distance $D_{\mathcal{R}}^\xi: X\rightarrow \mathbb{R}^+$ of two element $x_1, x_2\in X$ is defined by
    \[  D_{\mathcal{R}}^\xi(x_1, x_2):= \mathcal{R}(x_1)-\mathcal{R}(x_2)-\langle \xi, x_1-x_2\rangle, \]
where $\xi\in \partial \mathcal{R}(x_2)$.
\end{definition}

\begin{definition}\label{def2.4}
{\rm  (Morozov's discrepancy principle)} Given $\tau_2>\tau_1\ge 1$, choose $\alpha=\alpha(\delta, y^{\delta})>0$ such that
\begin{equation}\label{equ2.4}
    \tau_1\delta\leq \|F(x_{\alpha}^{\delta})-y^{\delta}\|_Y\leq \tau_2 \delta
\end{equation}
holds for an element $x_{\alpha}^{\delta}\in \mathcal{L}^{\delta}_{\alpha}$.
\end{definition}

\par The following first order necessary condition for problem \eqref{equ1.2} is standard (\cite[Lemma 3.1]{DH20}).

\begin{lemma}\label{lem2.5}
If $\hat{x}$ is a minimizer of $\mathcal{J}_{\alpha}^{\delta}(x)$, then
\begin{equation}\label{equ2.5}
    \langle F'(\hat{x})(F(\hat{x})-{y^{\delta}}),z-\hat{x}\rangle\geq \alpha\mathcal{R}(\hat{x})-\alpha\mathcal{R}(z)\quad \forall\, z\in X.
\end{equation}
\end{lemma}

\section{Existence of regularization parameter}\label{sec3}

\par In this section, we discuss the existence of a regularization parameter $\alpha$ determined by MDP. In the rest of this paper, we view $\alpha$ as the regularization parameter.  For a fixed noise level $\delta>0$, define functions
\begin{align*}
    G(x_{\alpha}^{\delta}) &:=\frac{1}{2}\|F(x_{\alpha}^{\delta})-y^{\delta}\|_Y^{2},\\
    m(\alpha) &:=\mathcal{J}_{\alpha}^{\delta}(x_{\alpha}^{\delta})=\min \mathcal{J}_{\alpha}^{\delta}(x)
\end{align*}
for $\alpha\in (0, \infty)$.

\par Next, we recall some properties of $\mathcal R(x_{\alpha}^{\delta})$, $G(x_{\alpha}^{\delta})$ and $m(\alpha)$, cf.\ \cite{TLY1998} for their proofs.

\begin{lemma}\label{lem3.1}
$\mathcal R(x_{\alpha}^{\delta})$ is non-increasing, whereas $G(x_{\alpha}^{\delta})$ and $m(\alpha)$ are non-decreasing with respect to $\alpha\in (0, \infty)$.
\end{lemma}


\begin{lemma}\label{lem3_2} 
Under Condition \ref{con2.1}, there exist $\alpha_{1},\alpha_{2}\in\mathbb{R}^{+}$
such that
    \[ \|F(x_{\alpha_{1},\beta_{1}}^{\delta})-y^{\delta}\|_{Y} <\tau_1\delta <\tau_2\delta<\|F(x_{\alpha_{2},\beta_{2}}^{\delta})-y^{\delta}\|_{Y}, \]
where $\beta_1=\alpha_1\eta$, $\beta_2=\alpha_2\eta$ and $\eta\in [0,1]$.
\end{lemma}

%

\par Thanks to Conditions \ref{con2.1} and \ref{con2.2}, we have the following standard result, cf.\  \cite{EHN1996} for its proof.

\begin{lemma}\label{lem3.3}
Let $\alpha_{n}\rightarrow\alpha>0$ as $n\rightarrow\infty$. Denote by $x_{n}:=x_{\alpha_{n}}^{\delta}$ a minimizer of $\mathcal{J}_{\alpha_{n}}^{\delta}(x)$.
Then $\{x_{n}\}$ contains a convergent subsequence $\{x_{n_{k}}\}$
such that $x_{n_{k}}\rightarrow x_{\alpha}^{\delta}$ in $X$,
where $x_{\alpha}^{\delta}$ is a minimizer of $\mathcal{J}_{\alpha}^{\delta}(x)$.
\end{lemma}

\begin{proposition}\label{pro3.4}
The function $m(\alpha)$ is continuous with respect to $\alpha$.
\end{proposition}

\par Although $m(\alpha)$ is continuous, the functions $G(x_{\alpha}^{\delta})$ and $\mathcal{R}(x_{\alpha}^{\delta})$ are not necessarily continuous with respect to $\alpha$. If \eqref{equ2.4} does not hold for any parameter $\alpha$, then $\|F(x_{\alpha}^{\delta})-y^{\delta}\|_{Y}$ has a jump at a certain parameter. Actually, we have the following lemma, which can be proved similar to that of Theorem 2.3 in \cite{TLY1998}.

\begin{lemma}\label{lem3.5}
Denote $G(x):=\|F(x)-y^{\delta}\|_Y$. For each ${\alpha}>0$, there exist $x_{\alpha}^{\delta,1}, x_{\alpha}^{\delta,2}\in \mathcal{L}^{\delta}_{{\alpha}}$ such that
\begin{align*}
    G(x_{\alpha}^{\delta,1})=\inf\limits_{x\in \mathcal{L}^{\delta}_{{\alpha}}}G(x),\quad G(x_{\alpha}^{\delta,2})=\sup\limits_{x\in \mathcal{L}^{\delta}_{{\alpha}}}G(x).
\end{align*}
\end{lemma}

\par A minimizer of $\mathcal{J}_{\alpha}^{\delta}\left(x\right)$ may be non-unique. Hence, the inequality \eqref{equ2.4} may not hold. More precisely, there may exist a regularization parameter $\alpha$ such that
\begin{equation}\label{equ3.1}
\|F(x_{\alpha}^{\delta,1})-y^{\delta}\|_{Y}<\tau_1\delta<\tau_2\delta<\|F(x_{\alpha}^{\delta,2})-y^{\delta}\|_{Y}.
\end{equation}
To prove the existence of a regularization parameter $\alpha$ satisfying \eqref{equ2.4}, it is sufficient to ensure that there are no minimizers $x_{\alpha}^{\delta,1}$, $x_{\alpha}^{\delta,2}$ of $\mathcal{J}_{\alpha}^{\delta}(x)$ for \eqref{equ3.1} to be valid. It can be shown that if there is no parameter $\alpha$ such that \eqref{equ3.1} holds, then there exists $\alpha$ satisfying \eqref{equ2.4}. The proof is similar to that of Theorem 3.10 in \cite{AR10}.

\begin{lemma}\label{lem3.6}
Assume Conditions \ref{con2.1} and \ref{con2.2}.  If there is no $\alpha>0$ with minimizers $x_{\alpha}^{\delta,1},x_{\alpha}^{\delta,2}\in \mathcal{L}_{{\alpha}}^{\delta}$ such that \eqref{equ3.1} is valid, then there exist $\alpha=\alpha(\delta,y^{\delta})>0$ and $x_{\alpha}^{\delta}\in\mathcal{L}_{\alpha}^{\delta}$
such that \eqref{equ2.4} holds.
\end{lemma}

\par Next we prove an estimate for $\|F(x_{\alpha}^{\delta,2})-y^{\delta}\|_{Y}$ if $x_{\alpha}^{\delta,1}$, $x_{\alpha}^{\delta,2}$ satisfy \eqref{equ3.1}.

\begin{lemma}\label{lem3.7}
Assume $x_{\alpha}^{\delta,1}$, $x_{\alpha}^{\delta,2}$ satisfy \eqref{equ3.1}. Then
    \[\tau_2\delta<\|F(x_{\alpha}^{\delta,2})-y^{\delta}\|_{Y}\le\tau_1 \left(3+2\gamma\right) \delta.\]
\end{lemma}
\begin{proof}
Denote $x_{1}:=x_{\alpha}^{\delta,1}$,
$x_{2}:=x_{\alpha}^{\delta,2}$.
From Lemma \ref{lem2.5}, we have
\begin{equation}\label{equ3.2}
    \left\langle F'(x_1)^*\left(F(x_1)-y^{\delta}\right),z-x_1\right\rangle\geq\alpha\mathcal{R}(x_1)- \alpha\mathcal{R}(z)\quad\forall\, z\in X.
 \end{equation}
Thus,
    \[\left\langle F'(x_1)^*\left(F(x_1)-y^{\delta}\right), x_2-x_{1}\right\rangle\geq \alpha\mathcal{R}(x_1)-\alpha\mathcal{R}(x_2)\label{equnew2},\]
which can be rewritten as
\begin{align*}
\alpha\mathcal{R}(x_1)-\alpha\mathcal{R}(x_2) \le \|F(x_{1})-y^{\delta}\|_{Y}\|F'(x_{1})(x_2-x_1)\|_{Y}.
\end{align*}
It follows from Condition \ref{con2.1} (iii) that
\begin{align}\label{equ3.3}
    \alpha\mathcal{R}(x_1)-\alpha\mathcal{R}(x_2)\le  \|F(x_{1})-y^{\delta}\|_{Y}\left(1+\gamma\right)\|F(x_2)-F(x_1)\|_Y.
\end{align}
In addition,
\begin{align}\label{equ3.4}
\frac{1}{2}\|F(x_{2})-y^{\delta}\|_{Y}^{2} &
=\frac{1}{2}\|F(x_{2})-y^{\delta}\|_{Y}^{2}+\alpha\mathcal{R}(x_2)-\alpha\mathcal{R}(x_2)\nonumber\\
&=\frac{1}{2} \|F(x_{1})-y^{\delta}\|_{Y}^{2}+\alpha\mathcal{R}(x_1)-\alpha\mathcal{R}(x_2).
\end{align}
A combination of \eqref{equ3.3} and \eqref{equ3.4} leads to
\begin{align}\label{equ3.5}
    \frac{1}{2}\|F(x_{2})-y^{\delta}\|_{Y}^{2}
    &\le\frac{1}{2}\tau_1^2\delta^{2}+\tau_1\delta(1+\gamma)\|F(x_{2})-y^{\delta}+y^{\delta}-F(x_1)\|_{Y}\nonumber\\
    &\le\frac{2\gamma+3}{2}\tau_1^2\delta^{2}+\tau_1\delta(1+\gamma)\|F(x_2)-y^{\delta}\|_{Y},
\end{align}
which is rewritten as
\begin{align}\label{equ3.6}
    \|F(x_{2})-y^{\delta}\|_{Y}^{2}-(2+2\gamma)\tau_1\delta\|F(x_2)-y^{\delta}\|_{Y} -(2\gamma+3)\tau_1^2\delta^{2}\leq 0.
\end{align}
The inequality \eqref{equ3.6} implies that 
\[-\tau_1\delta\le\|F(x_{2})-y^{\delta}\|_{Y}\le(2\gamma+3)\tau_1\delta.\]
This proves the lemma.
\end{proof}

\par From Lemmas \ref{lem3.6} and \ref{lem3.7}, we deduce the following result.

\begin{theorem}\label{the3.8}
Assume Conditions \ref{con2.1} and \ref{con2.2}. Then there exist a regularization parameter $\alpha$ and a minimizer $x_{\alpha}^{\delta}\in\mathcal{L}_{{\alpha}}^{\delta}$ such that
\begin{align}\label{equ3.7}
    \tau_1\delta\le\|F(x_{\alpha}^{\delta})-y^{\delta}\|_{Y}\le c\,\delta, \quad c:=\max\left\{ \tau_2,(3+2\gamma)\tau_1 \right\}.
\end{align}
\end{theorem}
\begin{proof}
By Lemma \ref{lem3.7}, if there exists a regularization parameter $\alpha$ such that $x_{\alpha}^{\delta,1}$, $x_{\alpha}^{\delta,2}$ satisfy \eqref{equ3.1}, then
    \[\tau_2\delta<\|F(x_{\alpha}^{\delta,2})-y^{\delta}\|_{Y}\leq \tau_1(3+2\gamma)\delta. \]
This implies that there exists a regularization parameter $\alpha$ such that \eqref{equ3.7} holds.
On the contrary, if there is no regularization parameter $\alpha$ with minimizers $x_{\alpha}^{\delta,1}$, $x_{\alpha}^{\delta,2}$ such that \eqref{equ3.1} is valid, by Lemma \ref{lem3.6}, there exist $\alpha$ and $x_{\alpha}^{\delta}\in\mathcal{L}_{\alpha}^{\delta}$
such that \eqref{equ2.4} holds. Obviously, \eqref{equ2.4} implies \eqref{equ3.7}. Hence, there exists a regularization parameter $\alpha$ such that \eqref{equ3.7} holds, which proves the theorem.
\end{proof}

\par In the following, we will show that for MDP, $\alpha\equiv \alpha(\delta,y^{\delta})\rightarrow 0$ as the noise level $\delta\rightarrow 0$.

\begin{lemma}\label{lem3.9}
Let $\alpha>0$ be fixed. If $F(x^*)=y$ and
    \[x^*\in \mathop{\arg\min}_{x\in X}\left\{\frac{1}{2}\|F(x)-y\|_Y^2+\alpha\mathcal{R}(x)\right\},\]
then $x^*=0$.
\end{lemma}
\begin{proof}
By the assumption, we have
\begin{align}\label{equ3.8}
    \alpha\mathcal{R}(x^*)&=\frac{1}{2}\|F(x^*)-y\|_Y^2+\alpha\mathcal{R}(x^*)\nonumber\\
    & \le \frac{1}{2}\|F((1-t)x^*)-y\|_Y^2+\alpha\mathcal{R}((1-t)x^*)\quad\forall\,t\in (0,1).
\end{align}
Since $\mathcal{R}((1-t)x^*)=(1-t)\mathcal{R}(x^*)$, we can rewrite \eqref{equ3.8} as
\[ 0\le \alpha\mathcal{R}(x^*) \le \frac{1}{2}\frac{\|F((1-t)x^*)-y\|_Y^2}{t}\quad\forall\,t\in (0,1).\]
Since
    \[ F((1-t)x^*)-y = F(x^*)+F'(x^*)(-tx^*)+o(\|-tx^*\|_{\ell_2})-y = F'(x^*)(-tx^*)+o(t),\]
we have
    \[ 0\le \alpha\mathcal{R}(x^*) \le \frac{t}{2}\,\|F'(x^*)x^*+o(t)\|_Y^2\quad\forall\,t\in (0,1).\]
Then take the limit as $t\to 0^+$ to obtain $\mathcal{R}(x^*)=0$.  By Condition \ref{con2.2}, $\mathcal{R}(x^*)=0$ implies that $x^*=0$.
\end{proof}

\begin{theorem}\label{the3.10}
Let $\delta_n\rightarrow 0$ and $y^{{\delta}_n}\rightarrow y$ as $n\rightarrow \infty$.  Let $\alpha_n:=\alpha(\delta_n, y^{\delta_n})$ be the regularization parameter obtained from MDP \eqref{equ3.7} with $\delta$ replaced by ${\delta_n}$.  Then $\alpha_n\rightarrow 0$ as $n\rightarrow \infty$.
\end{theorem}
\begin{proof}
We argue by contradiction.  Suppose $\alpha_n$ does not converge to 0, i.e.,
$\exists\,\alpha_0>0$, $\forall\,N\in \mathbb{N}^+$, $\exists\,n_0>N$ such that $|\alpha_{n_0}-0|\ge \alpha_0$.
This implies that there exists a subsequence of $\{\alpha_n\}$, still denoted by $\{\alpha_n\}$ such that $\alpha_n\ge\alpha_0$.
Denote
    \[x^{\alpha_0}_n\in \mathop{\arg\min}_{x\in X}\left\{\frac{1}{2}\|F(x)-y^{\delta_n}\|_Y^2+\alpha_0\mathcal{R}(x) \right\}\]
and
    \[x_n:=x_{\alpha_n}^{\delta_n}\in \mathop{\arg\min}_{x\in X}\left\{\frac{1}{2}\|F(x)-y^{\delta_n}\|_Y^2+{\alpha_n}\mathcal{R}(x) \right\}.\]
By Lemma \ref{lem3.1}, $\|F(x_{\alpha}^{\delta})-y^{\delta}\|_Y^2$ is non-decreasing with respect to $\alpha$. Hence,
\begin{align}\label{equ3.9}
    \frac{1}{2}\|F(x^{\alpha_0}_n)-y^{\delta_n}\|_Y^2&\le\frac{1}{2}\|F(x_n)-y^{\delta_n}\|_Y^2\nonumber\\
    & \le \frac{1}{2}\max\{{\tau_2}^2\delta_n^2,(3+2\gamma)^2{\tau_1}^2\delta_n^{2}\}\rightarrow 0.
\end{align}
By the definition of $x^{\alpha_0}_{n}$, there exist $x^*\in X$ and a subsequence $\{x^{\alpha_0}_{n_k}\}$ such that $x^{\alpha_0}_{n_k}\rightharpoonup x^*$ in $X$. Since $F$ is weakly sequently closed and $y^{{\delta}_n}\rightarrow y$,
$F(x^{\alpha_0}_{n_k})-y^{\delta_{n_k}}\rightharpoonup F(x^*)-y$, it follows from the weak lower semi-continuity of the norm that
\begin{align}\label{equ3.10}
0&\le \frac{1}{2}\|F(x^*)-y\|_Y^2\le \liminf\limits_{k\rightarrow\infty}\frac{1}{2}\|F(x^{\alpha_0}_{n_k})-y^{\delta_{n_k}}\|_Y^2\nonumber\\
 & \le \frac{1}{2}\max\{\tau_2^2\delta_{n_k}^2,(3+2\gamma)^2\tau_1^2\delta_{n_k}^{2}\}\rightarrow 0.
\end{align}
Hence, $F(x^*)=y$. In addition,
\begin{align}\label{equ3.11}
\frac{1}{2}\|F(x^*)-y\|_Y^2+\alpha_0\mathcal{R}(x^*)&\le \liminf\limits_{k\rightarrow \infty} \bigg\{\frac{1}{2}\|F(x^{\alpha_0}_{n_k})-y^{\delta_{n_k}}\|_Y^2+\alpha_0\mathcal{R}(x^{\alpha_0}_{n_k})\bigg\}\nonumber\\
& \le \liminf\limits_{k\rightarrow\infty}\bigg\{\frac{1}{2}\|F(x)-y^{\delta_{n_k}}\|_Y^2+\alpha_0\mathcal{R}(x)\bigg\}\nonumber\\
& = \frac{1}{2}\|F(x)-y\|_Y^2+\alpha_0\mathcal{R}(x)
\end{align}
for any $x\in X$. Thus, $x^*$ is a minimizer of $\frac{1}{2}\|F(x)-y\|_Y^2+\alpha_0\mathcal{R}(x)$. By Lemma \ref{lem3.9}, this implies that $x^*=0$. Then
$y=F(0)$. Hence, $\|F(0)-y^{\delta}\|_Y=\|y-y^{\delta}\|_Y\le \delta$, contradicting to $\|F(0)-y^{\delta}\|>\delta$ in Condition \ref{con2.1}\,(iv). This proves the theorem.
\end{proof}

\par For a numerical realization of MDP \eqref{equ3.7}, we can make use of an iterative algorithm described in \cite{R02}, cf.\ Algorithm \ref{alg1}. Even though $\alpha$ is determined from the upper bound: $\|F(x_{\alpha}^{\delta})-y^{\delta}\|_Y=c\delta$, we can still obtain the convergence of the regularized solution. Numerical simulation results reported in Section \ref{sec3} show that we can obtain better inversion results if $\alpha$ is determined by Algorithm \ref{alg1}.

\begin{algorithm}
\caption{Iterative algorithm for $\alpha$ under MDP \eqref{equ3.7}}
\label{alg1}
\begin{algorithmic}
\STATE{Choose $\tau>0$, $\eta=1$, $0<q<1$, $j=0$, $\alpha_0>0$ with $\|F(x_{\alpha_0}^{\delta})-y^{\delta}\|_Y>c\delta$ }
\STATE{for $j=1,2,\cdots$,}
\STATE{if $\|F(x_{\alpha_{j-1}}^{\delta})-y^{\delta}\|_Y>c\delta$}
\STATE{~~~~$\alpha_j=q\alpha_{j-1}$}
\STATE{~~~~compute $x_{\alpha_j}^{\delta}$}
\STATE{~~~~$\alpha_{j}^{\max}=\alpha_{j-1}$, $\alpha_{j}^{\min}=\alpha_{j}$}
\STATE{else}
\STATE{~~~~$\alpha_j=\left(\alpha_{j-1}^{\min}+\alpha_{j-1}^{\max}\right)/2$}
\STATE{~~~~compute $x_{\alpha_j}^{\delta}$} by iterative soft thresholding algorithm
\STATE{~~~~if $\|F(x_{\alpha_j}^{\delta})-y^{\delta}\|_Y> c\delta$ then $\alpha_{j}^{\max}=\alpha_j$}
\STATE{~~~~if $\|F(x_{\alpha_j}^{\delta})-y^{\delta}\|_Y< \tau_1\delta$ then $\alpha_{j}^{\min}=\alpha_j$}
\STATE{end}
\end{algorithmic}
\end{algorithm}

\par We note that the upper bound $c\delta$ from Theorem \ref{the3.8} is $O(\delta)$. Then the convergence of the regularized solution is standard (\cite{EHN1996}). Indeed, we can show the regularized solution $x_{\alpha}^{\delta}$ defined by (\ref{equ2.1}) converges to an $\mathcal R$-minimum solution
of problem $F(x)=y$.

\section{Regularization properties}\label{sec4}

\par In this section, we examine the well-posedness of the regularization method under MDP (\ref{equ3.7}). We prove that the regularized solution $x_{\alpha}^{\delta}$ defined by (\ref{equ2.1}) converges to an $\mathcal R$-minimum solution of the problem $F(x)=y$. In addition, we discuss the convergence rate of $x_{\alpha}^{\delta}$. The proofs are along the lines of standard quadratic Tikhonov regularization (\cite{EHN1996}) and sparsity regularization (\cite{GHS08,JLS09,SGGHL2009,WLMC13}). However, analysis of the convergence rate is different from that in \cite{DH19} since the regularization parameter $\alpha$ is now determined by \eqref{equ3.7}.

\begin{theorem}\label{the4.1} 
{\rm (Convergence)} Let $x_{\alpha_{n}}^{\delta_{n}}$ be a minimizer of $\mathcal{J}_{\alpha_n}^{\delta_n}(x)$ defined by \eqref{equ1.2} with the data $y^{\delta_n}$ satisfying $\| y-y^{\delta_n}\|\leq\delta_n$, where $\delta_n\rightarrow 0$ as $n\rightarrow \infty$ and $y^{\delta_n}$ belongs to the range of $F$. Let $\alpha_n$ be chosen by the Morozov's discrepancy principle \eqref{equ3.7}. Then there exists a subsequence of $\{x_{\alpha_{n}}^{\delta_{n}}\}$, still denoted by $\{x_{\alpha_{n}}^{\delta_{n}}\}$, that converges to an $\mathcal{R}$-minimizing solution $x^{\dag}$ in $X$. In addition, if the $\mathcal{R}$-minimizing solution $x^{\dag}$ is unique, then
    \[\lim\limits_{n \rightarrow \infty} \| x_{\alpha_{n}}^{\delta_{n}}-x^{\dag}\|_{X}=0.\]
\end{theorem}
\begin{proof}
Denote $y_{n}:=y^{\delta_{n}}$, $x_{n}:=x_{\alpha_{n}}^{\delta_{n}}$. By the definition of $x_{n}$, we obtain
\begin{align}\label{equ4.1}
    \frac{1}{2}\| F(x_{n})-y_{n}\|_Y^{2} +\alpha_{n}\mathcal{R}(x_{n})
    &\leq\frac{1}{2}\|F(x^{\dag})-y_n\|_Y^2+\alpha_{n}\mathcal{R}(x^{\dag})\nonumber\\
    &\leq\frac{1}{2}\delta_n^2+\alpha_{n}\mathcal{R}(x^{\dag}).
\end{align}
By \eqref{equ3.7}, this implies that $\mathcal{R}(x_n)\le \mathcal{R}(x^{\dag})$. Hence, the sequence $\{\mathcal{R}(x_n)\}$ is bounded. Denote
    \[c:=\max\left\{ \tau_2,(3+2\gamma)\tau_1 \right\}.\]
We have $c\delta_n\rightarrow 0$ as $\delta_n\rightarrow 0$. In addition,
\begin{equation}\label{equ4.2}
    \| F(x_{n})-y\|_Y\leq\| F(x_{n})-y_{n}\|_Y+\| y-y_{n}\|_Y\leq c\delta_n+\delta_{n}.
\end{equation}
Then
\begin{equation}\label{equ4.3}
\begin{array}{llc}
    \lim\limits_{n \to \infty}F(x_n)=y.
\end{array}
\end{equation}
On the other hand, it follows from (\ref{equ4.1}) that
\begin{equation}\label{equ4.4}
    \limsup\limits_{n \to \infty}\mathcal{R}(x_n)
    \leq\mathcal{R}(x^{\dag}).
\end{equation}
Since $ \mathcal{R}(x_n)$ is bounded, there exist an element $x^{*}\in X$
and a subsequence of $\{x_n\}$, still denoted by $\{x_n\}$, such that $x_{n}\rightharpoonup x^{*}$ in $X$.
Together with (\ref{equ4.3}), it follows that
    \[\|F(x^*)-y\|_Y\le \liminf\limits_{n \to \infty}\|F(x_n)-y\|_Y=0.\]
Hence, $F(x^*)=y$. Meanwhile, by Condition \ref{con2.2} (ii), we have
\begin{align}\label{equ4.5}
 \mathcal{R}(x^*)\leq\liminf\limits_{n}\mathcal{R}(x_n)\leq\mathcal{R}(x^{\dag}).
 \end{align}
By the definition of $x^{\dag}$, $x^{*}$ is an $\mathcal{R}$-minimizing solution. If the $\mathcal{R}$-minimizing solution is unique, then $x^{*}=x^{\dag}$. A combination of (\ref{equ4.4}) and (\ref{equ4.5}) implies $\mathcal{R}(x_n)\rightarrow\mathcal{R}(x^{\dag})$. Thus, $\mathcal{R}(x_n)\rightarrow \mathcal{R}(x^{\dag})$. Then $\lim\limits_{n \rightarrow \infty} \| x_{n}-x^{\dag} \|_{X}=0$ by Condition \ref{con2.2} (iii).
\end{proof}

\par Next, we present a result on the linear convergence rate $O(\delta)$ of the regularized solution $x_{\alpha}^{\delta}$ under the source condition and Bregman distance.  For this purpose, we introduce an assumption on an $\mathcal R$-minimizing solution $x^{\dagger}$ of the problem $F(x)=y$.

\begin{assumption}\label{ass4.2}
There exists an $\omega\in Y$ such that
\begin{equation}\label{equ4.6}
    F'(x^{\dag})^{*}\omega\in \partial\mathcal{R}(x^{\dag}).
\end{equation}
\end{assumption}

\begin{theorem}\label{the4.3} 
{\rm (Convergence rate)} Let $ x^{\delta}_{\alpha}$ be defined by \eqref{equ2.1}. Assume that there exist parameters $\alpha$ satisfying MDP \eqref{equ3.7}. Then, under Assumptions \ref{ass4.2}, there exists $d\in D_{\mathcal{R}}^\xi(x_{\alpha}^{\delta}, x^{\dag})$ such that $d=O(\delta)$.
\end{theorem}
\begin{proof} 
Due to MDP (\ref{equ3.7}) and the minimization property of $\displaystyle x^{\delta}_{\alpha}$,
\begin{align*}
\frac{1}{2}\delta^{2}+\alpha \mathcal{R}(x_{\alpha}^{\delta})
 & \leq\frac{1}{2}\|F(x_{\alpha}^{\delta})-y^{\delta}\|_{Y}^{2}+\alpha \mathcal{R}(x_{\alpha}^{\delta})\\
    & \leq\frac{1}{2}\|F(x^{\dagger})-y^{\delta}\|_{Y}^{2}+\alpha \mathcal{R}(x^{\dagger})\\
    & \leq\frac{1}{2}\delta^{2}+\alpha\mathcal{R}(x^{\dagger}).
\end{align*}
Then $\mathcal{R}(x_{\alpha}^{\delta})\leq \mathcal{R}(x^{\dagger})$. Consider the Bregman distance 
\begin{equation*}
    d= \mathcal{R}(x_{\alpha}^{\delta})-\mathcal{R}(x^{\dagger})-\langle F'(x^{\dagger})^{*}\omega, x_{\alpha}^{\delta}-x^{\dagger}\rangle.
\end{equation*}
Since $\mathcal{R}(x_{\alpha}^{\delta})\leq \mathcal{R}(x^{\dagger})$,
\begin{equation}\label{equ4.7}
    d\leq -\langle F^{'}(x^{\dagger})^{*}\omega,x_{\alpha}^{\delta}-x^{\dagger}\rangle\leq \|\omega\|\|F^{'}(x^{\dagger})(x_{\alpha}^{\delta}-x^{\dagger})\|_{Y}.
\end{equation}
A combination of \eqref{equ4.7} and the nonlinear condition \eqref{equ2.2} implies that
\begin{align*}
    d
    & \leq \|\omega\|(1+\gamma)\|F(x_{\alpha}^{\delta})-F(x^{\dagger})\|_{Y}\nonumber\\
    & \leq \|\omega\|(1+\gamma)(\|F(x_{\alpha}^{\delta})-y^{\delta}\|_{Y}+\|y-y^{\delta}\|_{Y}).
\end{align*}
By MDP \eqref{equ3.7}, $\|F(x_{\alpha}^{\delta})-y^{\delta}\|_{Y}\le c\delta$. Thus
    \[d \leq \|\omega\|(1+\gamma)(c+1)\delta,\]
i.e., $d=O(\delta)$.
\end{proof}

\section{Numerical experiments}\label{sec5}

\par In this section, we present two numerical examples to demonstrate the efficiency of MDP \eqref{equ3.7}. The relative error (Rerror) is utilized to measure the accuracy of the reconstruction $x^{*}$:
    \[  \mathrm{Rerror}:=\frac{\|x^{*}-x^{\dag}\|}{\|x^{\dag}\|}, \]
where $x^{\dag}$ is a true solution. The exact data $y^{\dag}$ is generated by $y^{\dag}=F(x^{\dag})$. The noise level $\delta$ is defined by $\delta:=\|F(x^{\dag})-y^{\delta}\|_2$. All numerical calculations were done using MATLAB R2010 on an i7-6500U 2.50GHz workstation with 8Gb RAM.

\subsection{A nonlinear compressive sensing problem}

First, we solve a nonlinear compressive sensing (CS) problem (\cite{BE13,B13,CL16,YWLJWJ15}) to demonstrate the efficiency of MDP for the $\ell_{1}$-regularization. The research of nonlinear CS is important not only in theoretical analysis but also in many applications where the observation system is often nonlinear. For example, in diffraction imaging, a charge coupled device (CCD) records the amplitude of the Fourier transform of the original signal. So one only obtains the nonlinear measurements of the original signal. Fortunately, as shown in \cite{B13}, if the system satisfies some nonlinear conditions then recovery is still possible. Under the nonlinear CS frame, the measurement system is nonlinear. Assume, therefore, that the observation model is
\begin{equation}\label{equ5.1}
    y=F(x)+\delta,
\end{equation}
where $F: \mathbb{R}^n\rightarrow \mathbb{R}^m$ is a nonlinear operator, $x\in \mathbb{R}^n$ and $\delta\in \mathbb{R}^m$ is a noise. It can be shown that if the linearization of $F$ at an exact solution $x^{\dag}$ satisfies the restricted isometry property (RIP), then the convergence property of the iterative hard thresholding algorithm (IHTA) is guaranteed (\cite{B13}). Next we illustrate the efficiency of the proposed algorithm by a nonlinear CS example of the form (\cite{B13})
\begin{equation}\label{equ5.2}
y=F(x):=\hat{a}(A\hat{b}(x)),
\end{equation}
where $A$ is a CS matrix, $\hat{a}$ and $\hat{b}$ are nonlinear operators, respectively. Here, $\hat{a}$ encodes nonlinearity after mixing by $A$ as well as nonlinear ``crosstalk'' between mixed elements, and $\hat{b}$ encodes the same system properties for the inputs before mixing. For simplicity, we write $\hat{a}(x)= x^c$ and $\hat{b}(x)= x+b(x)$, where $b$ is a nonlinear map. In particular, let $b(x)=x^d$, where $c, d \in \mathbb{N}^{+}$ and $x^c$ and $x^d$ should be understood in a componentwise sense.

\par Next, we consider the nonlinearity of the operator $F$. In \cite{B13}, it is shown that the Jacobian matrix of $\hat{a}(A\hat{b}(x))$ is of the form
    \[F'(x)=[a'_x(\Phi\hat{b}(x))][\Phi(I+b'_x(x))].\]
There exists a constant $\gamma>0$ such that
\[ \|F(x_1)-F(x_2)-F'(x_2)(x_1-x_2)\|_{Y}\leq L\|x_1-x_2\|_{\ell_2},\]
cf.\ \cite[Lemmas 3 \& 4]{B13}. For example, we let $c=1$, $d=3$, then $F(x)=A(x+x^3)$. Assume $A$ is bounded.  Then obviously,
\[ \|x_1-x_2\|_{\ell_2}\le K\|F(x_1)-F(x_2)\|_Y.\]
Consequently, $F$ satisfies the nonlinear condition \eqref{con2.2}.

\par In the experiment, an ST-($\alpha\ell_{1}-\beta\ell_{2}$) ($\beta=0$) algorithm for nonlinear ill-posed problems is used to compute the iterative solution, cf.\ \cite{DH24} for details of the algorithm.
The ST-($\alpha\ell_{1}-\beta\ell_{2}$) algorithm reduces to the form
\[ x^{k+1}=\mathbb{S}_{\alpha/\lambda} \left(x^k-\frac{1}{\lambda}F'(x^k)^*(F(x^k)-y^{\delta})\right). \]

\par Let $n = 200$, $m = 0.4n$, $p = 0.2m$, $\tau_1 = 1$ and $\tau_2 = 2$, where $p$ is the number of the impulses in the true solution. White Gaussian noise is added to the exact data $y^{\dag}$ by calling $\mathrm{y^{\delta}=awgn( F(x^{\dag}), \sigma})$ in MATLAB, where $\sigma$ is the added noise, measured in dB, which measures the ratio between the true (noise free) data $y^{\dag}$ or $F(x^{\dag})$ and Gaussian noise. Meanwhile, we let $\gamma=1/2$, $c=1$ and $d=3$. For the sparsity regularization of linear ill-posed problems,  the value of $\|A_{m\times n}\|_{\ell_{2}}$ needs to be less than 1 (\cite{DDD04}). This condition is also needed for the nonlinear CS problem \eqref{equ5.1}. Since the value of $\|A_{m\times n}\|_{\ell_{2}}$ is around 20, we re-scale the matrix $A_{m\times n}$ by $A_{m\times n}\rightarrow0.05A_{m\times n}$. The initial value $x^0=\mathrm{zeros}(n,1)$.

\begin{figure}[tbhp]
\centering
\includegraphics[width=180mm,height=30mm]{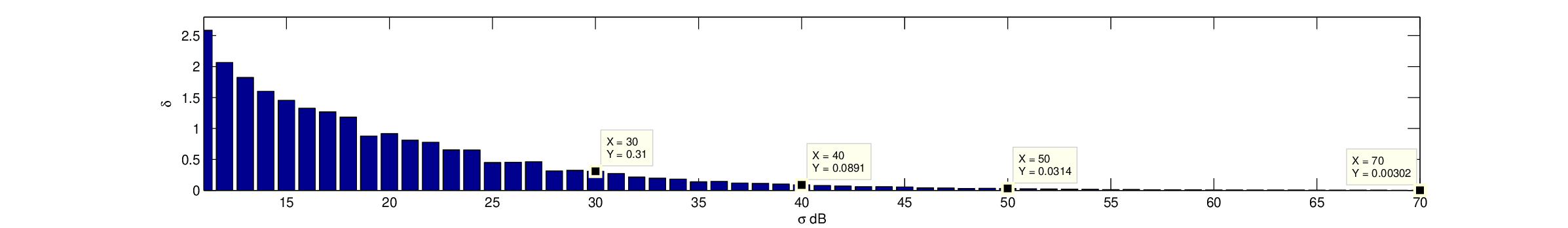}
\caption{The noise level $\delta$ vs.\ $\sigma$.}
\label{fig1}
\end{figure}

\begin{figure}[tbhp]
\centering
\includegraphics[width=180mm,height=30mm]{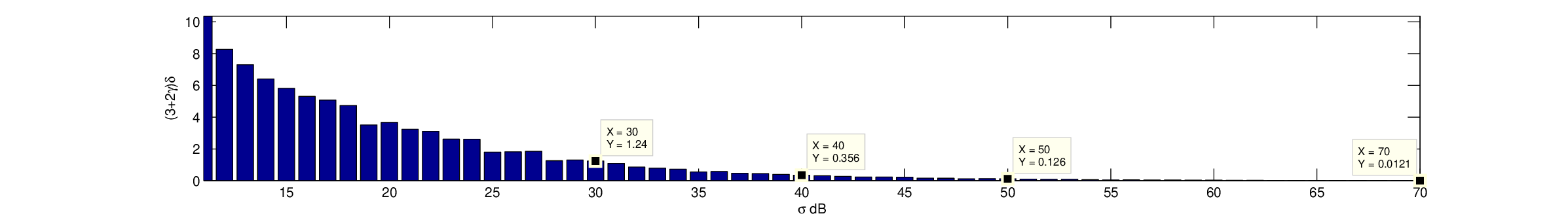}
\caption{The upper bound $c\delta$ vs.\ $\sigma$. }
\label{fig2}
\end{figure}

\par In Fig.\ \ref{fig1} and Fig.\ \ref{fig2}, we show relations among the noise level $\delta$, $c\delta$ and $\sigma$. Since the Gaussian noise is random, strict monotonicity is not maintained numerically for the noise level and the upper bound in \eqref{equ1.5} with respect to $\sigma$. Nevertheless, overall, the approximate monotonicity property is observed.

\begin{figure}[tbhp]
\centering
\subfigure[]{\includegraphics[width=80mm,height=60mm]{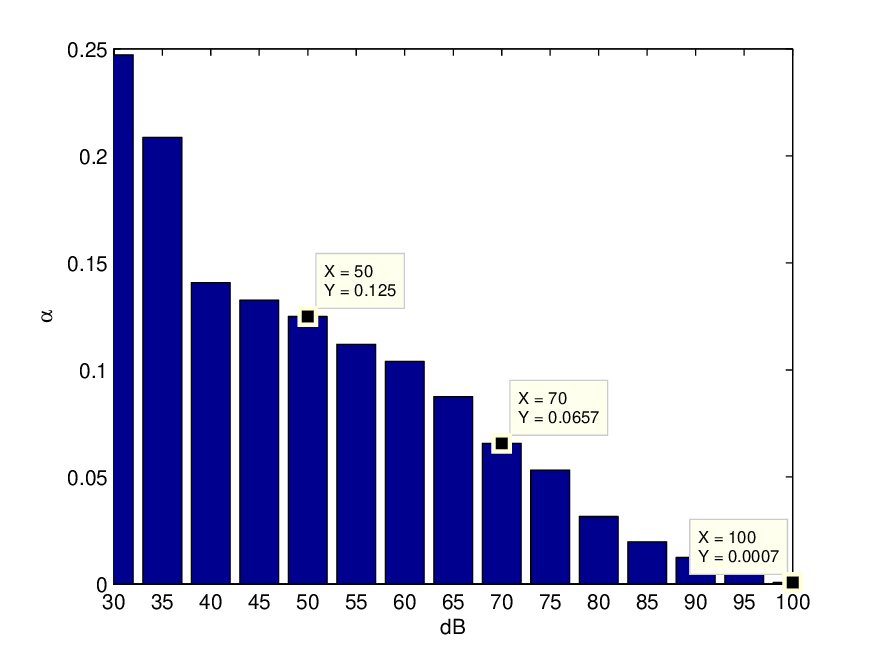}}
\subfigure[]{\includegraphics[width=80mm,height=60mm]{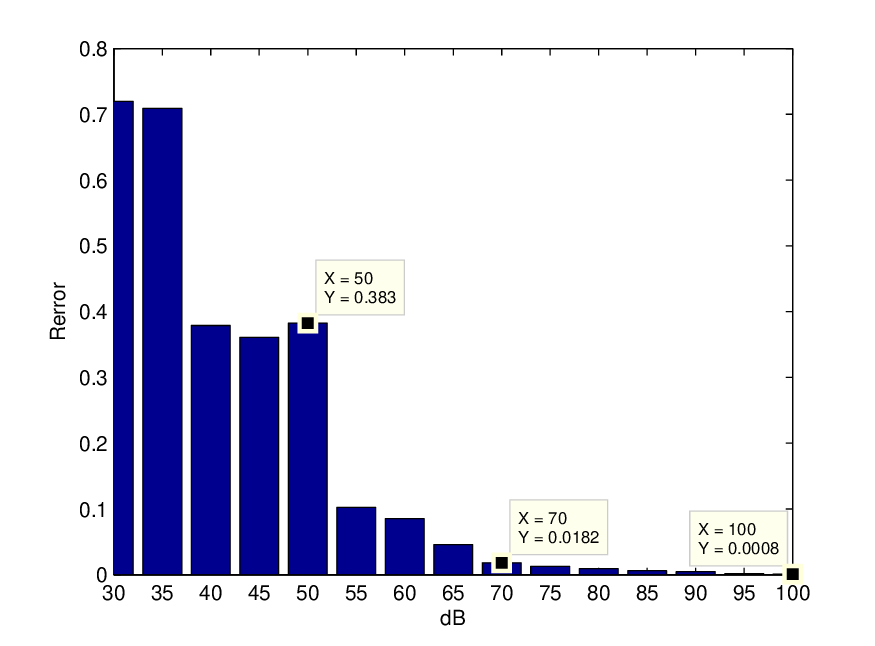}}
\caption{(a) $\alpha$ vs.\ $\sigma$, $\alpha$ being determined by $\|F(x_{\alpha, \beta}^{\delta})-y^{\delta}\|_Y=c\delta$; (b) Rerror vs.\ $\sigma$.}
\label{fig3}
\end{figure}

\begin{figure}[tbhp]
\centering
\subfigure[]{\includegraphics[width=80mm,height=60mm]{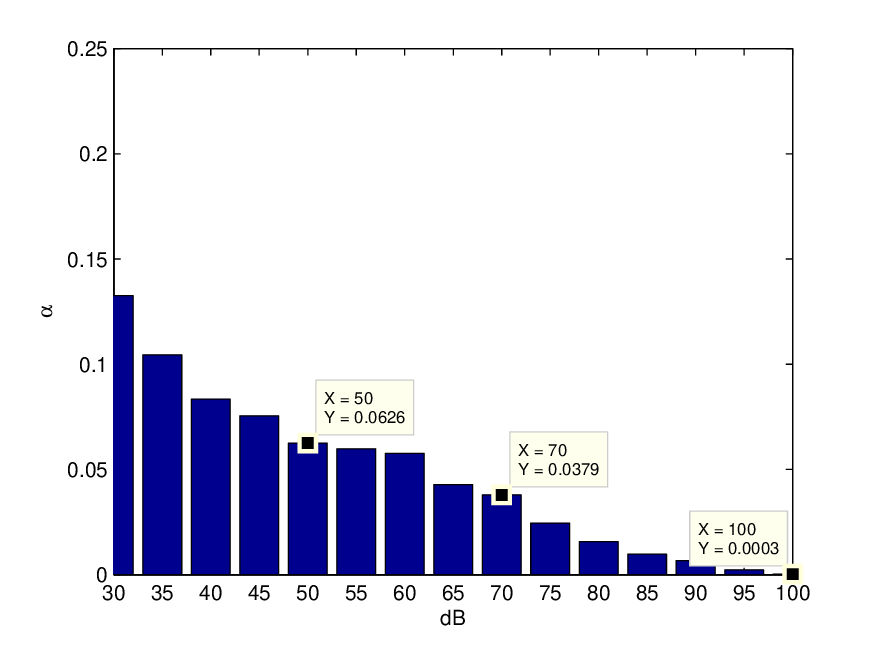}}
\subfigure[]{\includegraphics[width=80mm,height=60mm]{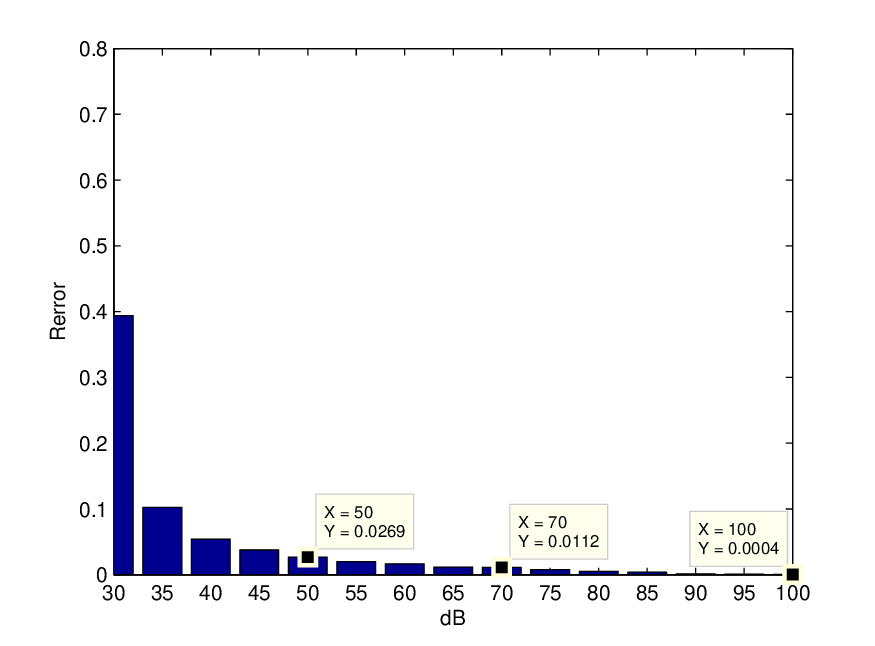}}
\caption{(a) $\alpha$ vs.\ $\sigma$, $\alpha$ being determined by Algorithm \ref{alg1}; (b) Rerror vs.\ $\sigma$.}
\label{fig4}
\end{figure}

\par To test sharpness of the upper bound in \eqref{equ1.5}, let $\tau_{1}=1$ and choose $\alpha$ with $\|F(x_{\alpha, \beta}^{\delta})-y^{\delta}\|_Y=c\delta$. The value of $\alpha$ satisfying $\|F(x_{\alpha, \beta}^{\delta})-y^{\delta}\|_Y= c\delta$ can be obtained only approximately. In computational practice, this value of $\alpha$ can be found approximately by a sorting from a given set of values $\alpha_1, \alpha_2, \cdots, \alpha_n$. For example, with a good initial guess $\alpha_0$, we let $\alpha_n=\alpha_0+0.001*n$ (or $\alpha_n=\alpha_0-0.001*n$) to find a satisfactory approximate value of $\alpha$ of a solution of $\|F(x_{\alpha, \beta}^{\delta})-y^{\delta}\|_Y= c\delta$. In Fig.\ \ref{fig3} (a), we observe that $\alpha\rightarrow 0$ as the noise level $\delta\rightarrow 0$. In Fig.\ \ref{fig3} (b), we see that Rerror decreases with respect to $\sigma$. This illustrates the convergence of the regularized solution with respect to the noise level $\delta$. Although the upper bound in \eqref{equ1.5} appears to be loose, if the noise levels are small enough, we can still recover satisfactory results.

\par Next, we test the convergence of the regularized solutions where $\alpha$ is determined by Algorithm \ref{alg1}. In Fig.\ \ref{fig4}, it is shown that $\alpha\rightarrow 0$ as the noise level $\delta\rightarrow 0$ and Rerror converges to 0 as $\delta\rightarrow 0$.

\begin{figure}[tbhp]
\centering
\subfigure[True solution]{\includegraphics[width=80mm,height=60mm]{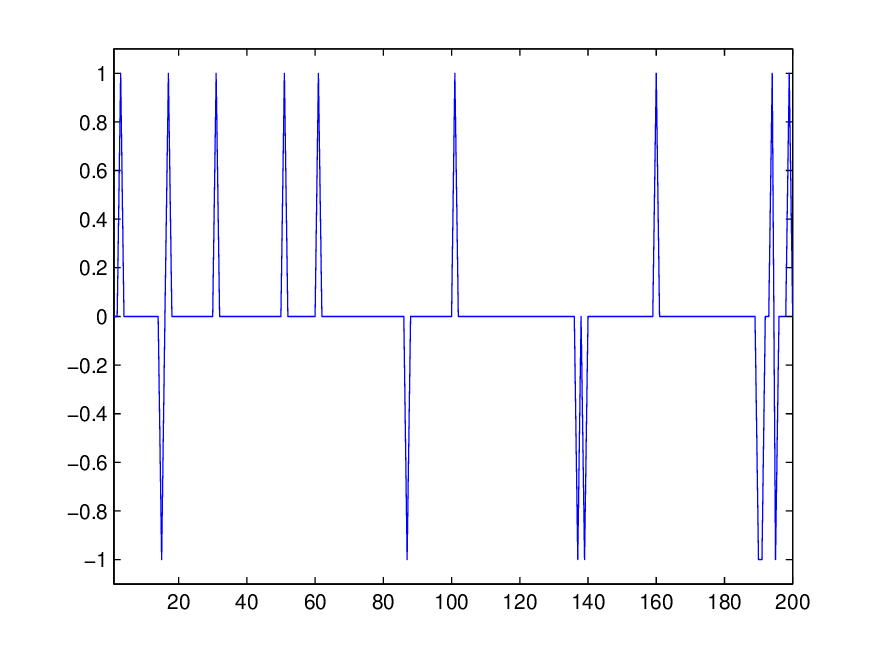}}
\subfigure[Observed data]{\includegraphics[width=80mm,height=60mm]{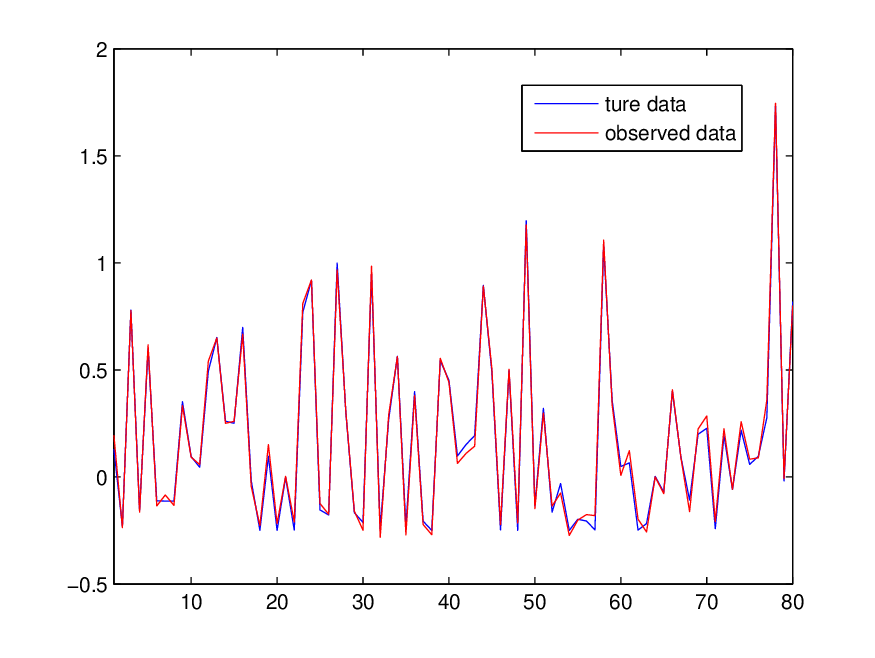}}
\subfigure[Rerror= 0.0445]{\includegraphics[width=80mm,height=60mm]{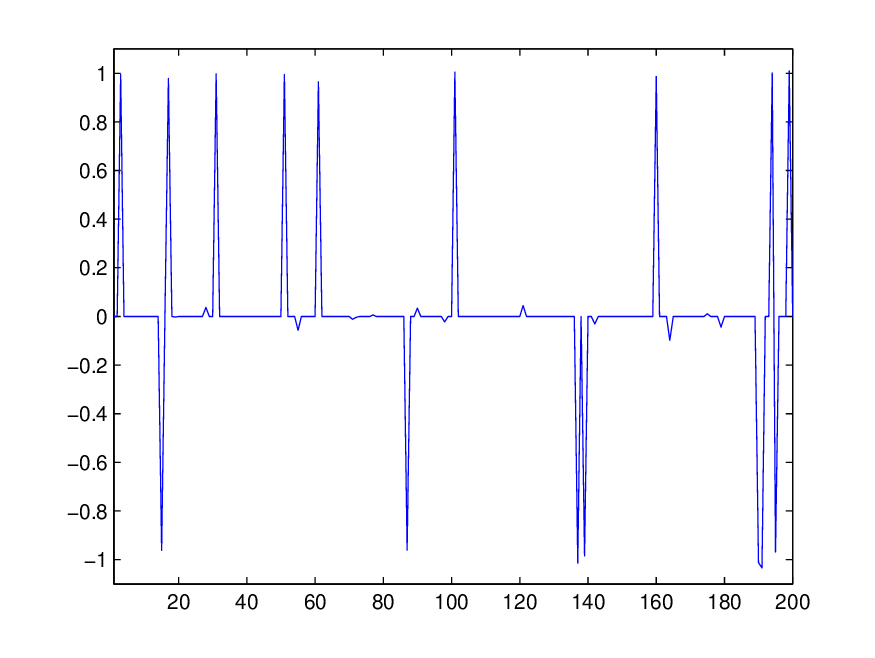}}
\subfigure[Rerror= 0.0217]{\includegraphics[width=80mm,height=60mm]{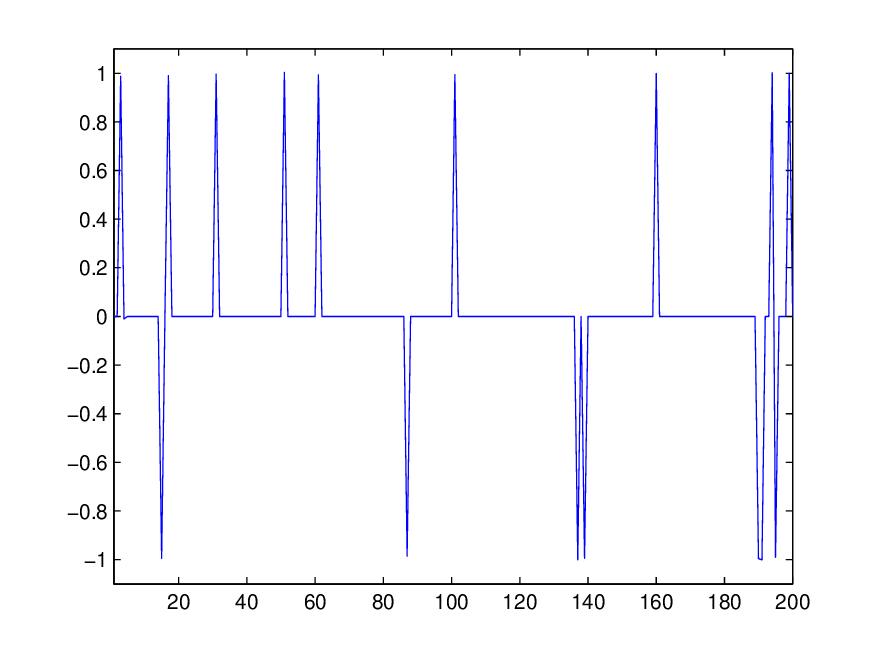}}
\caption{(a) True solution; (b) Observed data; (c) Recovered solution with $\alpha: \|F(x_{\alpha}^{\delta})-y^{\delta}\|_Y= c\delta$; (d) Recovered solution with $\alpha $ is determined by Algorithm \ref{alg1}.}
\label{fig5}
\end{figure}

\par Graphs of the reconstruction $x^{*}$ corresponding to $\sigma=60$dB ($\delta=0.014$ and $c\delta=0.056$) can be seen in Fig.\ \ref{fig5}. In Fig.\ \ref{fig5} (c), the reconstruction $x^{*}$ is computed with $\alpha: \|F(x_{\alpha}^{\delta})-y^{\delta}\|_Y= c\delta$. Rerror is 4.45\%. In Fig.\ \ref{fig5} (d), the reconstruction $x^{*}$ is computed with $\alpha$ determined by Algorithm \ref{alg1}. Rerror is 2.17\%. It is shown that we can obtain better results if the parameter $\alpha$ is determined by Algorithm \ref{alg1}.

\begin{figure}[tbhp]
\centering
\subfigure[$\tau_{1}=1$, $\tau_{1}=2$, $\left(2\gamma+3\right)\tau_{1}=4$ ]{\includegraphics[width=80mm,height=60mm]{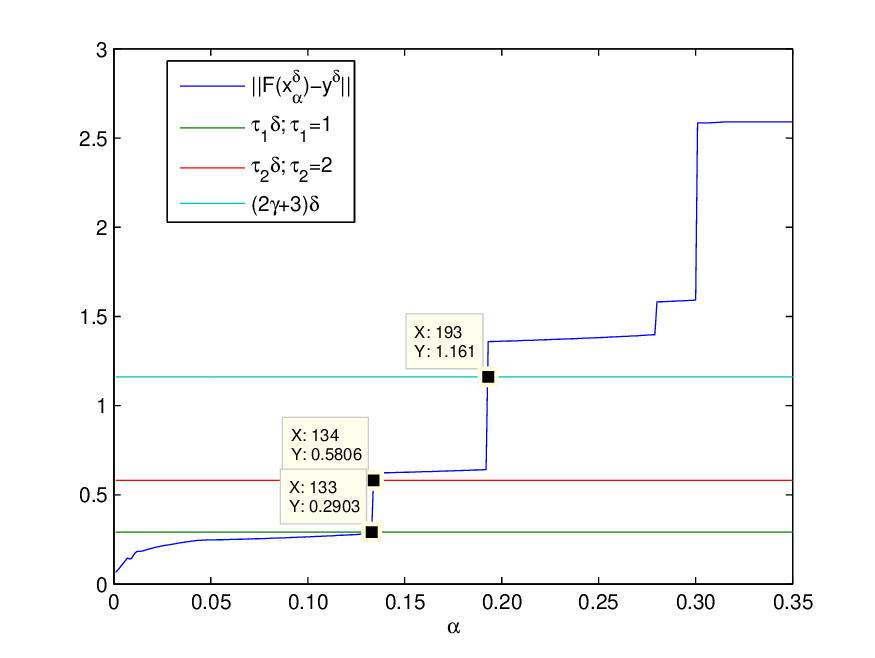}}
\subfigure[$\tau_{1}=2.208$, $\tau_{1}=4.681$, $\left(2\gamma+3\right)\tau_{1}=8.832$]{\includegraphics[width=80mm,height=60mm]{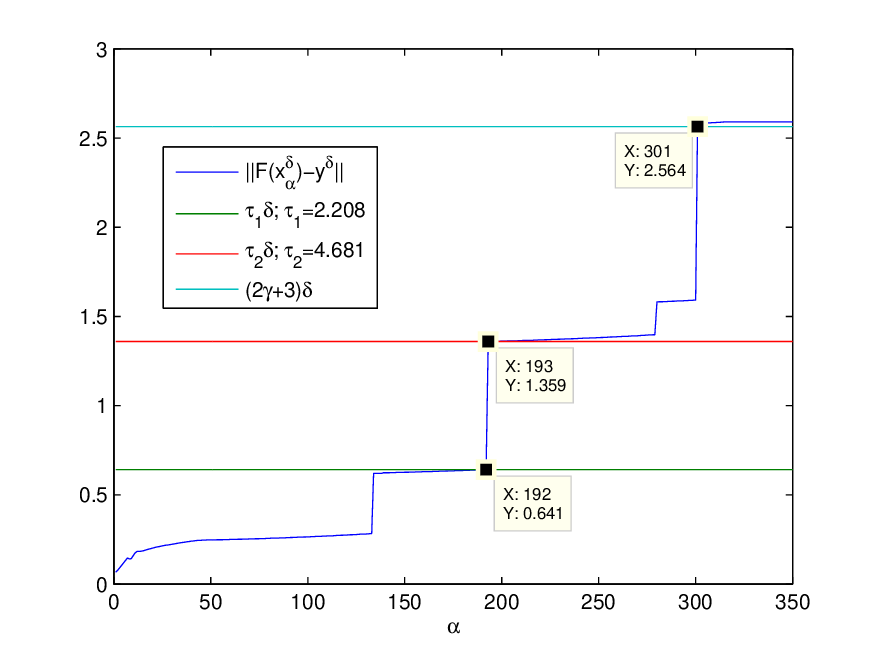}}
\caption{The value of the discrepancy $\left\|F\left(x\right)-y^{\delta}\right\|_{Y}$ by the ST-$\left(\alpha\ell_{1}-\beta\ell_{2}\right)$ $\left(\beta=0\right)$ algorithm with different values of $\alpha$.}
\label{fig6}
\end{figure}

\par Next we show if $\tau_{1}$ and $\tau_{2}$ are chosen small, the existence of $\alpha$ can not be guaranteed under the condition \eqref{equ1.4}. We let $\sigma=30$dB, $\delta=0.2903$ and $c\delta=1.1612$. In Fig.\ \ref{fig6} (a), it is shown that if $1\leq \tau_{1}<\tau_{2}\leq 2$, then $\tau_{1}\delta\geq 0.5823$ and $\tau_{2}\delta\leq 0.5823$. Consequently, there is no $\alpha$ such that the traditional MDP \eqref{equ1.4} holds. Even with larger values, e.g. $2.208 \leq\tau_{1}<\tau_{2}\leq 4.681$, it is challenging to find $\alpha$ such that \eqref{equ1.4} holds, as shown in Fig.\ \ref{fig6} (b). Indeed, for the traditional MDP with linear ill-posed problems, one commonly tries $\alpha_j=\alpha\,2^{-j}$, $j=1, 2, \cdots$. We increase $j$ and calculate $x_{\alpha_j}^\delta$ until \eqref{equ1.4} is valid.

\begin{figure}[tbhp]
\centering
\includegraphics[width=80mm,height=60mm]{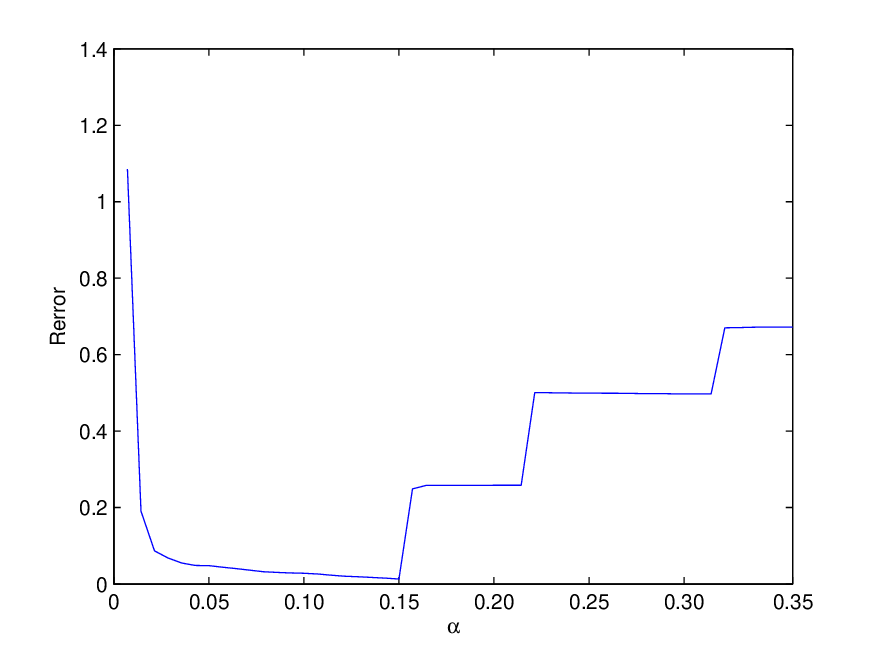}
\caption{The relative error of reconstruction $x^*$ by the ST-$\left(\alpha\ell_{1}-\beta\ell_{2}\right)$ $\left(\beta=0\right)$ algorithm with different $\alpha$.}
\label{fig7}
\end{figure}

\par However, for nonlinear ill-posed problems, the discrepancy is not continuous with respect to $\alpha$. If $\alpha_j=0.2$, then $\alpha_{j+1}=0.1$, we can not find an $\alpha$ such that $\tau_{1}\delta\leq \|F(x_{\alpha}^{\delta})-y^{\delta}\|_Y\leq \tau_{2}\delta$ holds (cf.\ Fig.\ \ref{fig6}). Certainly, we can try $\alpha_j=\alpha\,3^{-j}$ or $\alpha_j=\alpha\,4^{-j}$, etc. Nevertheless, it requires more computational time. In addition, one still does not have the theoretical assurance on the existence of $\alpha$. So it is challenging to choose appropriate values of $\tau_{1}$ and $\tau_{2}$ to ensure the existence of $\alpha$. Although the upper bound $c(\delta)$ in \eqref{equ1.5} seems loose, it guarantees the existence of $\alpha$. We can determine an $\alpha$ by Algorithm \ref{alg1} such that the modified MDP \eqref{equ1.5} holds. So the upper bound $c(\delta)$ appears to be natural.

\par Finally, Fig.\ \ref{fig7} illustrates the relationship between the relative error and the regularization parameter. The diagram indicates that the relative error of the reconstructed solution is minimized when $\alpha=0.15$. At this point, the regularization parameter $\alpha$ satisfies $\tau_{1}\delta\leq \left\|F(x_{\alpha}^{\delta})-y^{\delta}\right\|_{Y}\leq c\delta$, but it does not fulfill $\tau_{1}\delta\leq \left\|F(x_{\alpha}^{\delta})-y^{\delta}\right\|_{Y}\leq \tau_{2}\delta$, as shown in Fig.\ \ref{fig6} (a). This further emphasizes that a more suitable regularization parameter can be identified using Algorithm \ref{alg1}.

\subsection{A nonlinear one-dimensional gravity inversion problem}

\par We then present results on a nonlinear one-dimensional gravity inversion problem \cite{RAUCAH24} to demonstrate the efficiency of MDP for the Tikhonov regularization. For this one-dimensional gravity inversion problem, the operator $F$ exhibits Lipschitz continuity for its first-order derivative but fails to meet the tangential cone condition. The primary objective of this example is to assess the applicability of the method discussed in this paper to other nonlinear inversion problems, even when these problems do not adhere to the tangential cone condition.

Investigating one-dimensional gravity anomaly inversion is of significant importance not only for theoretical analyses of potential fields but also for its practical applications in geophysical exploration, where the nonlinear effects of complex geological structures frequently influence observational data. For instance, in the context of geophysical exploration, gravimeters measure the vertical component of the gravitational field \cite{B96}, which is induced by subsurface density anomalies. Due to the inherent characteristics of signal superposition and attenuation, where the field signal diminishes inversely with the square of the distance, actual observational data can be conceptualized as a nonlinear integral transform of the underlying density distribution. Thus, we can only extract one-dimensional nonlinear projection information regarding the subsurface media. Within the sparse constraint inversion proposed by \cite{LO98}, high-precision inversion remains attainable if the density anomalies exhibit local sparsity conditions, such as those characteristics of a finite layer model. The forward modeling of gravity anomalies is grounded in the gravitational theory about spherical bodies. Specifically, for a sphere of radius $R$, exhibiting a density contrast $\triangle\rho$, relative to its surrounding matrix and located at a depth $d$,  the vertical gravity anomaly recorded at the surface measurement point $x$ is given by
\begin{equation}\label{equ5.3}
    g(x)=\frac{4\pi}{3} \frac{G \triangle\rho R^{3} d} {\left(d^{2}+\left(x-x_{0}\right)^{2}\right)^{3/2}},
\end{equation}
where $G$ denotes the gravitational constant $(6.674\times 10^{-11} m^{3} kg^{-1} s^{-2})$, and $x_{0}$ represents the horizontal position of the sphere. In the case of two spheres, the cumulative gravity anomaly experienced at a given point is the summation of the anomalies produced by each sphere
\begin{equation}\label{equ5.4}
    g_{total}(x)=g_{1}(x)+g_{2}(x),
\end{equation}
where $g_{1}\left(x\right)$ denote the gravitational influence of the first spherical body at point $x$, while $g_{2}\left(x\right)$ signifies the gravitational influence of the second spherical body at the same point $x$. These functions collectively represent the contributions to the gravity anomaly at $x$ induced by each respective sphere. For this analysis, we examine a simplified scenario of one-dimensional gravity inversion, where two spheres contribute to the gravity anomalies, as illustrated in Fig.\ \ref{fig7}. The spheres under consideration possess a radius of $r_{1}=r_{2}=100m$, a density contrast of $\triangle\rho = 300 kg/m^{3}$, and depths from the surface to their centers of $d_{1} = 150m$ and $d_{2} = 200m$, respectively. This configuration can be interpreted as a model representing iron ore (a mineral characterized by high iron content) situated within sedimentary rock formations. Our objective is to ascertain the positional and depth-related information of the spheres based on the observed gravity anomaly data.

\begin{figure}[tbhp]
\centering
\includegraphics[width=150mm,height=100mm]{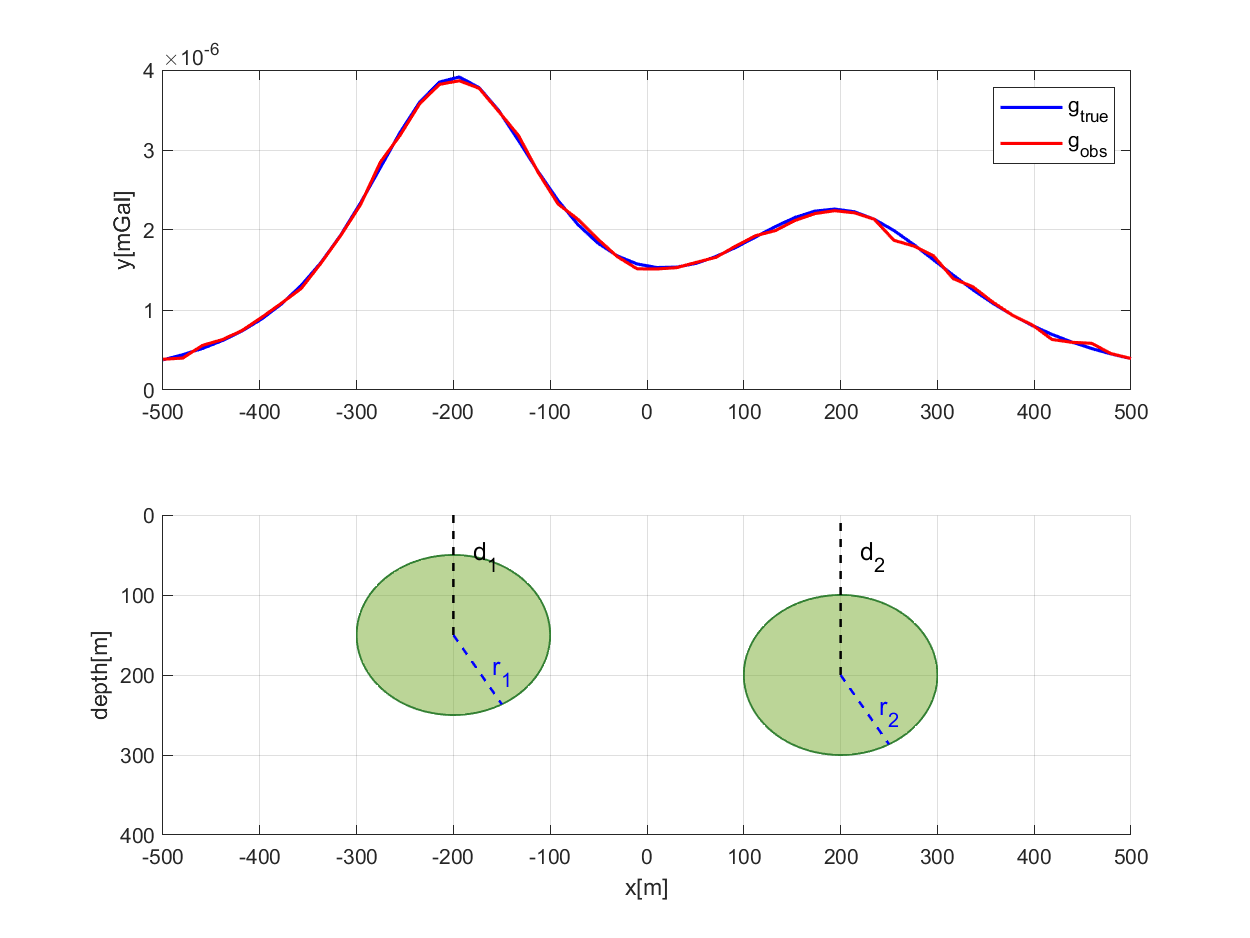}
\caption{Comparative analysis of observed Gravity values and theoretical values.}
\label{fig8}
\end{figure}

\par Next, we reformulate the problems above into the context of nonlinear ill-posed problems. We denote this relationship as
\begin{equation*}
    F\left(x\right)=g^{\delta},
\end{equation*}
where $F$ is a nonlinear operator associated with the gravity system, $g^{\delta}$ represents the observed gravity anomaly data, and $x$ is the parameter to be determined, which includes information on the location and the depth. 

\par We employ the Landweber iterative algorithm in the experiment to resolve these nonlinear ill-posed problems, as detailed in \cite{HNS95}. The Landweber iterative algorithm can be expressed in a simplified form
\begin{equation*}
    x^{k+1}=x^{k+1}-\omega_{k}F'(x^{k})^{*}\left(F(x^{k})-g^{\delta}\right).
\end{equation*}
To enhance the stability of the algorithm, we incorporate a Tikhonov regularization term \cite{EHN1996, FM99, IJ15} into the functional. To prevent the solution from getting trapped in local minima and to accelerate the convergence, we apply hard-thresholding to the regularization parameter and adapt the step size using the Frobenius norm of the Jacobian matrix \cite{AS21, BL08}.

\par We consider a scenario with $50$ measurement points, a side line range of $\left[-500,500\right]$, and a true parameter set of $\left[-200, 150, 200, 200\right]$ representing the positions and depths of two spheres. We establish a maximum iteration $maxiter=500$, an initial parameter guess of $\left[-150, 100, 150, 250\right]$, and introduce a 2\% Gaussian noise to simulate the disturbances typically encountered in actual observational processes.

\begin{figure}[tbhp]
\centering
\subfigure[Rerror=0.1195, $x^{*}$=\text{[-194.2, 148.2, 179.3, 213.7]}]{\includegraphics[width=80mm,height=60mm]{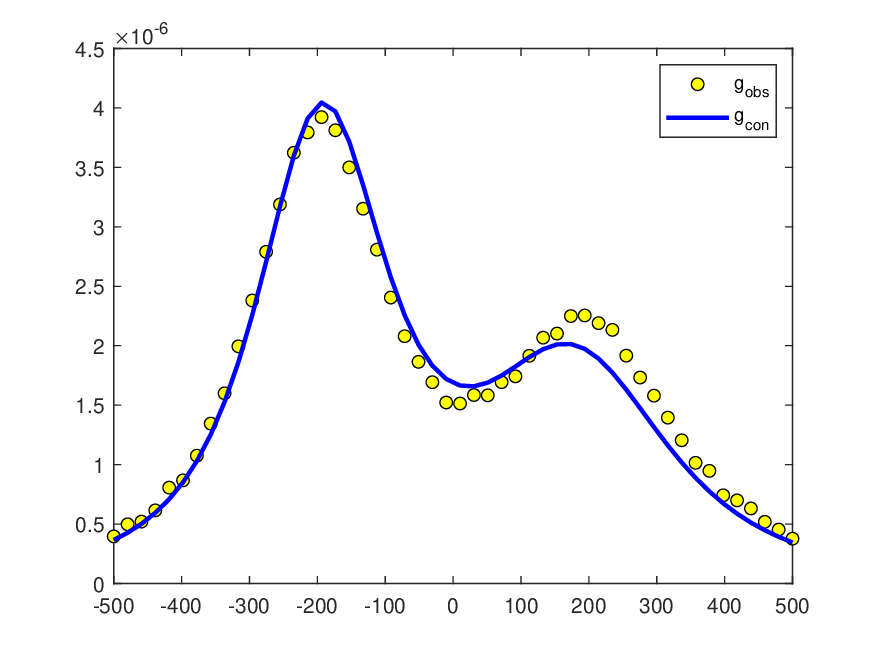}}
\subfigure[Rerror=0.0294, $x^{*}$=\text{[-199.1, 149.6, 193.6, 203.6]}]{\includegraphics[width=80mm,height=60mm]{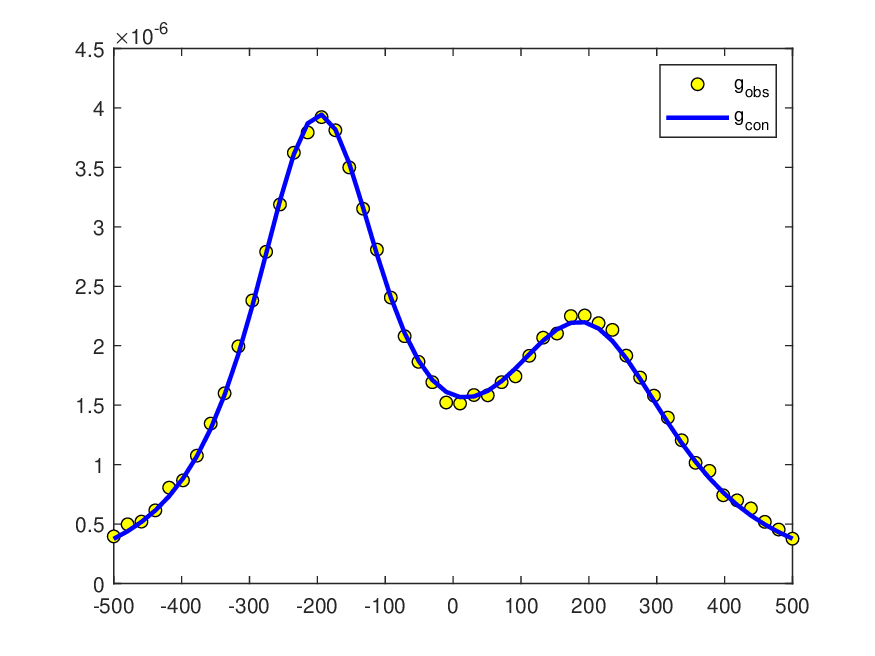}}
\caption{(a) Reconstructed data by recovered solution with $\alpha: \|F(x_{\alpha}^{\delta})-y^{\delta}\|_Y= c\delta$; (b) Reconstructed data by recovered solution with $\alpha$ is determined by Algorithm \ref{alg1}.}
\label{fig9}
\end{figure}

\par The theoretical measurement data, actual measurement data with the 2\% Gaussian noise ($\delta=2.7421\times 10^{-7}$ and $c\delta=1.2087\times 10^{-6}$), and the true location parameters of the two underground spheres are illustrated in Fig.\ \ref{fig8}. Fig.\ \ref{fig9} shows the theoretical measurement data generated from the reconstruction solutions obtained using various methods for selecting the regularization parameter. In Fig.\ \ref{fig9} (a), the reconstruction $x^{*}$ is calculated with $\alpha: \|F(x_{\alpha}^{\delta})-y^{\delta}\|_Y= c\delta$ and Rerror is 11.95\%. In Fig.\ \ref{fig9} (b), the reconstruction $x^{*}$ is derived using $\alpha$ determined by Algorithm \ref{alg1} and Rerror is 2.94\%. This demonstrates that better outcomes can be achieved when the parameter $\alpha$ is selected using Algorithm \ref{alg1}.

\begin{figure}[tbhp]
\centering
\subfigure[$\tau_{1}=1.102$, $\tau_{2}=2.771$, $\left(2\gamma+3\right)\tau_{1}=4.408$]{\includegraphics[width=80mm,height=60mm]{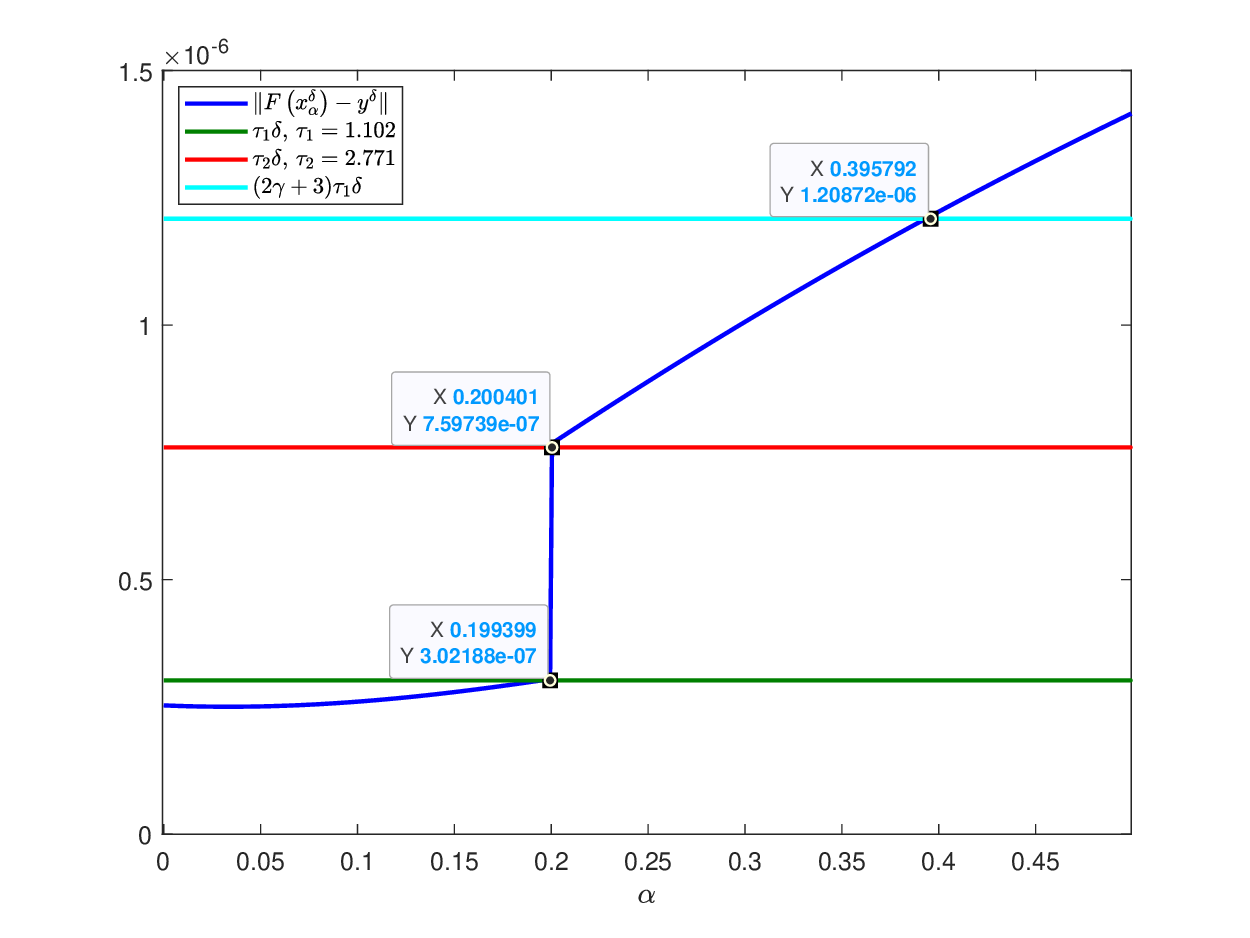}}
\subfigure[]{\includegraphics[width=80mm,height=60mm]{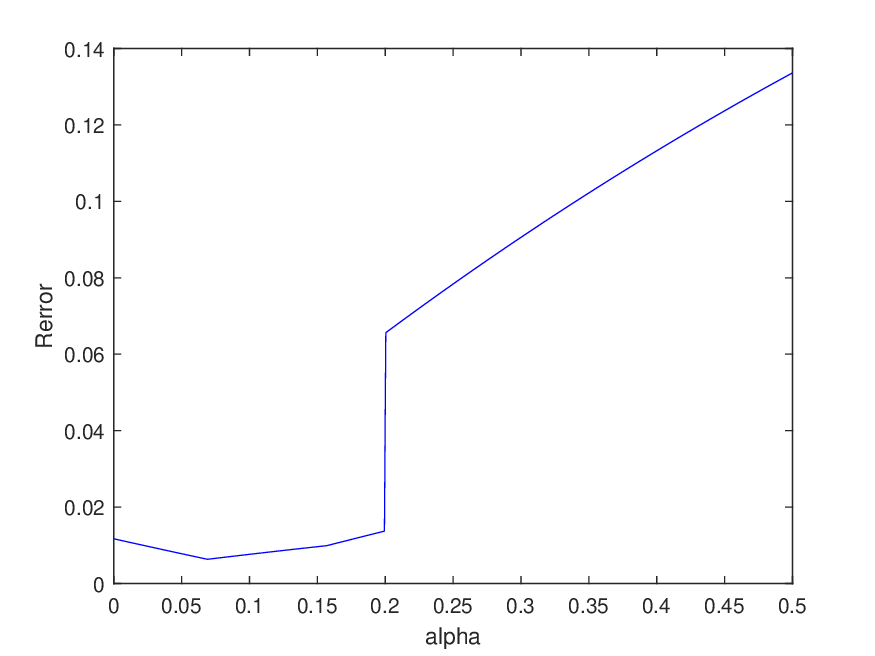}}
\caption{(a) The value of the discrepancy $\left\|F\left(x\right)-y^{\delta}\right\|_{Y}$ by the Landweber iterative algorithm with different $\alpha$; (b) The relative error of reconstruction $x^*$ by the Landweber iterative algorithm with different $\alpha$.}
\label{fig10}
\end{figure}

\par Additionally, we analyze the scenario where a small value is chosen for $\tau_{2}$. In this case, the existence of $\alpha$ can not be guaranteed under condition \eqref{equ1.4}. Maintaining the 2\% Gaussian noise yields $\delta=2.7421\times 10^{-7}$ and $c\delta=1.2087\times 10^{-6}$. As shown in Fig.\ \ref{fig10} (a), if $1.102\leq\tau_{1}<\tau_{2}\leq 2.771$, then $\tau_{2}\delta<7.5974\times 10^{-7}$. Consequently, no $\alpha$ satisfies the traditional MDP \eqref{equ1.4}. This observation further emphasizes that the discrepancy may not be continuous with respect to $\alpha$ for some nonlinear ill-posed problems. The upper bound $c\delta$ in \eqref{equ1.5} guarantees the existence of $\alpha$. We can determine an appropriate $\alpha$ using Algorithm \ref{alg1} such that the modified MDP \eqref{equ1.5} holds.

\par Finally, Fig.\ \ref{fig10} (b) illustrates the relationship between the relative error and the regularization parameter. The diagram indicates that the relative error of the reconstructed solution is minimized when $\alpha\approx 0.07$. At this point, though the regularization parameter $\alpha$ does not satisfy both $\tau_{1}\delta\leq \left\|F(x_{\alpha}^{\delta})-y^{\delta}\right\|_{Y}\leq c\delta$ and $\tau_{1}\delta\leq \left\|F(x_{\alpha}^{\delta})-y^{\delta}\right\|_{Y}\leq \tau_{2}\delta$, as shown in Fig.\ \ref{fig10}, Fig.\ \ref{fig9} (a) still emphasizes that a suitable regularization parameter can be identified by Algorithm \ref{alg1}.

\section{Conclusion}\label{sec7}

\par For Tikhonov regularization involving a nonlinear operator and a general convex penalty, we demonstrate the existence of a regularization parameter $\alpha$ such that
    \[ \tau_1\delta\le\|F(x_{\alpha}^{\delta})-y^{\delta}\|_{Y}\le c\delta,\]
where $c:=\max\left\{ \tau_2,(3+2\gamma)\tau_1\right\}$. Furthermore, we establish that $\alpha\equiv \alpha(\delta,y^{\delta})\rightarrow 0$ as the noise level $\delta\rightarrow 0$. It is shown that this regularization method is well-posed when $\alpha$ is determined using Morozov’s Discrepancy Principle. We also prove the convergence of the regularized solution $x_{\alpha}^{\delta}$, as defined by \eqref{equ2.1}, to an $\mathcal{R}$-minimum solution of the problem $F(x)=y$. This convergence occurs under Morozov’s principle, yielding a linear convergence rate $O(\delta)$ for the regularized solution $x_{\alpha}^{\delta}$ under the source condition and utilizing Bregman distance. Numerical simulation results illustrate the effectiveness of the regularization parameter selection method, showing that the relative error decreases as $\delta$ decreases. Although the upper bound in \eqref{equ1.5} may appear loose, the existence of $\alpha$ is assured. Moreover, we can still achieve satisfactory results if the noise levels are sufficiently low.

\end{document}